\documentclass{amsart}
\usepackage{amsfonts,dsfont}
\usepackage{amsmath}
\usepackage{amssymb}
\usepackage[a4paper,top=1.8 cm,bottom=1.8 cm,left=1.4cm,
right=1.4cm]{geometry}
\usepackage{color}
\usepackage{nccmath} 
\usepackage[shortlabels]{enumitem} 

\usepackage[hidelinks]{hyperref}
\usepackage{cite}

\newtheorem{theorem}{Theorem}
\newtheorem{corollary}[theorem]{Corollary}
\newtheorem{lemma}[theorem]{Lemma}
\newtheorem{proposition}[theorem]{Proposition}

\newtheorem{remark}[theorem]{Remark}

\theoremstyle{definition}
\numberwithin{equation}{section}
\numberwithin{theorem}{section}

\newcommand{\gI}{\mbox{GId}}

\newcommand{\F}{F\langle X\rangle}
\newcommand{\W}{W\langle X\rangle}

\newcommand{\spn}{\mbox{span}}
\newcommand{\UT}{UT_2}

\DeclareMathOperator{\sgn}{sgn}

\DeclareMathOperator{\gv}{\mbox{gvar}}
\DeclareMathOperator{\Id}{Id}
\DeclareMathOperator{\GId}{GId}

\DeclareMathOperator{\md}{mod}

\usepackage{xcolor}

\newcommand{\black}{\color{black}}

\begin{document}
	
	\title{The $2\times 2$ upper triangular matrix algebra and its generalized polynomial identities}
	
	\author[F.~Martino]{Fabrizio Martino}
	\address{Fabrizio Martino, Dipartimento di Matematica e Informatica, Università degli Studi di Palermo, via Archirafi 34, 90123, Palermo, Italy}
	\email{fabrizio.martino@unipa.it}
	
	\author[C.~Rizzo]{Carla Rizzo}
	\address{Carla Rizzo, Dipartimento di Matematica e Informatica, Università degli Studi di Palermo, via Archirafi 34, 90123, Palermo, Italy, \&  CMUC, Departamento de Matemática, Universidade de Coimbra, Largo D. Dinis, 3001-501, Coimbra, Portugal}
	\email{carlarizzo@mat.uc.pt}

    \thanks{This work was partially supported by GNSAGA - INDAM. The first author was partially supported by FFR2023 - Fabrizio Martino.
    The second author was partially supported by the Centre for Mathematics of the University of Coimbra - UIDB/00324/2020, funded by the Portuguese Government through FCT/MCTES}

	
	
	
	\begin{abstract}
		Let $UT_2$ be the algebra of $2\times 2$ upper triangular matrices over a field $F$ of characteristic zero. Here we study the generalized polynomial identities of $UT_2$, i.e., identical relations holding for $UT_2$ regarded as $UT_2$-algebra. We determine the generator of the $T_{UT_2}$-ideal of generalized polynomial identities of $UT_2$ and compute the exact values of the corresponding sequence of generalized codimensions. Moreover, we give a complete description of the space of multilinear generalized identities in $n$ variables in the language of Young diagrams through the representation theory of the symmetric group $S_n$. Finally, we prove that, unlike the ordinary case, the generalized variety of $UT_2$-algebras generated by $UT_2$ has no almost polynomial growth; nevertheless, we exhibit two distinct generalized varieties of almost polynomial growth.
	\end{abstract}
	
	\maketitle
	
	\section{Introduction}
	
	Let $A$ be an associative algebra over a field $F$ of characteristic zero, $\F$ be the free algebra generated by the countable set $X=\{x_1,x_2,\ldots \}$ and $W$ be a unitary associative algebra over $F$. Then $A$ is called $W$-algebra if it has a structure of $W$-bimodule with some additional conditions. A generalized polynomial identity of $A$ is a polynomial $f(x_1,\ldots, x_n)$ of the free $W$-algebra $W\langle X\rangle$ that vanishes under all substitutions of the elements of $A.$ Roughly speaking, $f(x_1,\ldots, x_n)$ is a polynomial of $\F$  with  ``coefficients" in $W.$ Notice that such ``coefficients" may appear also between two variables. Clearly, these identities are a natural generalization of the ordinary polynomial ones arising when $W$ coincides with $F$. The set of all generalized polynomial identities $\GId(A)$ is a $T_W$-ideal of $\W,$ i.e., an ideal stable by endomorphisms of $\W,$ and one of the main problems is to find a set of generators of such $T_W$-ideal.
	
	The idea of generalized polynomial identities stems from the observation that sometimes when we study polynomials in matrix algebras, we want to focus on evaluations where certain variables are always replaced by specific elements. Therefore, it would be useful to have a theory that allows us to consider ``polynomials” whose coefficients can be taken from an algebra, instead of from a field. 
	
	Generalized identities first appeared in 1965 in Amitsur’s fundamental paper \cite{Amitsur1965} on primitive rings satisfying generalized polynomial identities. In 1969, Martindale developed this idea further and applied it to prime rings \cite{Martindale1969}. 
	Later, two generalizations were pursued: Martindale \cite{Martindale1972} and Rowen \cite{Rowen1975,Rowen1976,Rowen1977} investigated generalized polynomial identities involving involutions, while Kharchenko \cite{Kharchenko1975,Kharchenko1978,Kharchenko1979} explored generalized polynomial identities involving derivations and automorphisms. These two directions were further developed and studied by various authors (see \cite{BeidarMartindaleMikhalevm1996} and its bibliography).
	In recent years, in case $W=A$ is finite dimensional and the bimodule action is the natural left and right multiplication, Gordienko in \cite{Gordienko2010} proved the so-called Amitsur conjecture, i.e., the limit $
	\lim_{n\rightarrow +\infty} \sqrt[n]{gc_n(A)},
	$
	where $gc_n(A)$, $n\geq 1$, is the generalized codimension sequence, exists and is a non-negative integer called the generalized PI-exponent of $A.$ He also proved that the generalized exponent equals the ordinary one defined by mean of the ordinary codimension sequence $c_n(A).$ For what concern the general
	the problem of describing the concrete generalized identities of an algebra so far it has been achieved 
	only for the algebra $M_n(F)$ of $n \times n$
	full matrices for all $n\geq 1$ (see for example \cite{BresarSpenko2015}). 
	
	The codimension sequence of an algebra was introduced by Regev in \cite{Regev1972} and it measures the rate of growth of the multilinear polynomials lying in the corresponding $T$-ideal. In the same paper, Regev proved that if $A$ satisfies a nontrivial polynomial identity, i.e., it is a PI-algebra, then its codimension sequence $c_n(A)$, $n\geq 1$, is exponentially bounded. Later Kemer in \cite{Kemer1979} showed that the variety generated by the algebra $UT_2$ of $2\times 2$-upper triangular matrices is of almost polynomial growth, i.e., it has exponential growth of the codimensions but every proper subvariety has polynomial growth. Analogous results were proved in various settings such as varieties of group-graded algebras \cite{Valenti2002}, algebras with derivation \cite{GiambrunoRizzo2019}, special Jordan algebras \cite{Martino2019}. It is worth mentioning that in the case of algebras with involution, Mishchenko and Valenti in \cite{MishchenkoValenti2000} constructed out of $UT_2$ a suitable algebra generating a variety of almost polynomial growth. 
	
	Motivated by the above results, here we deal with the generalized polynomial identities of $UT_2$ and we investigate the growth of the generalized codimension sequence $gc_n(A)$ of any algebra $A$ lying in the generalized variety generated by $UT_2.$
	
	The paper is organized as follows. After a necessary background on the generalized identities involving basic definitions and preliminary settings given in Section \ref{sezionedue}, we describe in Section \ref{sezionetre} the $T$-ideal of generalized identities of $UT_2$ as $UT_2$-algebra finding its generator. In Section \ref{sezionequattro} we study the space of multilinear generalized identities of $UT_2$ of degree $n$ as a representation of the symmetric group $S_n,$ decomposing its character into irreducibles by computing the corresponding multiplicities. Finally, in Section \ref{sezionecinque}, we prove the main result of the paper, i.e., the generalized variety of $UT_2$-algebras generated by $UT_2$, $\gv(UT_2),$ has no almost polynomial growth but we are able to construct inside $\gv(UT_2)$ a subvariety of almost polynomial growth. Moreover, we present another variety of $UT_2$-algebras of almost polynomial growth of the codimensions that is not contained in $\gv(UT_2)$.

	\section{On generalized polynomial identities and $W$-algebras}\label{sezionedue}
	
	Throughout this paper $F$ will denote a field of characteristic zero and all the algebras will be associative and have $F$ as their underlying field.

	Given an algebra $W$, we say that an algebra $A$ is a \emph{$W$-algebra}, if $A$ is a $W$-bimodule such that, for any $w\in W$, $a_1,a_2\in A$,
	\begin{equation}\label{conditions W-algebra}
		w(a_1 a_2)=(wa_1)a_2, \ (a_1 a_2)w=a_1(a_2 w), \ (a_1 w) a_2= a_1 (w a_2).
	\end{equation}
	When $W=F$, a $W$-algebra is just an $F$-algebra, that is an algebra over the field $F$.
	Clearly, $W$ itself has a natural structure of $W$-algebra by taking the left and right $W$-actions to be the usual left and right multiplications of $W$. In general, this is not the only way to define a structure of $W$-algebra on $W$ itself; in fact, there might exist different left and right $W$-actions on $W$ itself that induces a structure of $W$-algebra (see for example Section \ref{sezionecinque}).
	
	For fixed $W$ the class of $W$-algebras is a variety in the sense of universal algebra and is nontrivial since it contains $W$ itself. Ideals of $W$-algebras (\emph{$W$-ideals}) are understood to be invariant under the bimodule action of $W$, and homomorphisms $\varphi:A\rightarrow B$ between $W$-algebras $A,B$ must satisfy $\varphi(wav)=w\varphi(a)v$ for $a\in A$, $w,v\in W$.
	
	\smallskip

	The variety of $W$-algebras contains the \emph{free (associative) $W$-algebra} $\W$, freely generated by the countably infinite set of variables $X:=\{x_1,x_2, \dots\}$ which satisfies the following universal property: given a $W$-algebra $A$, any map $X\rightarrow A$ can be uniquely extended to a homomorphism of $W$-algebras $\W\rightarrow A$.

	We can give the following combinatorial description of  $\W$.
	First notice that it is not restrictive to assume that $W$ is an unital algebra; in fact, if not, we can consider the unital algebra $W^+=W+ F1$ obtained from $W$ by adding the unit element $1$.
	So, given a basis $\mathcal{B}_W:=\{w_i\}_{i\in\mathcal I}$ of $W$ such that $w_0=1$, if  we identify $x_i=1 x_i = x_i 1$ for $i\geq 1$, then a basis of $\W$ is the following
	$$
	\mathcal{B}_{\W}:=\left\lbrace w_{i_0}x_{j_1}w_{i_1}x_{j_2}\cdots w_{i_{n-1}}x_{j_n}w_{i_{n}} \ | \ n\geq 1, \, j_1,\dots,j_n\geq 1,\, w_{i_0},\ldots,w_{i_{n}}\in\mathcal{B}_W\right\rbrace .
	$$
	The multiplication of two elements $
	w_{i_0}x_{j_1}w_{i_1}x_{j_2}\cdots w_{i_{n-1}}x_{j_n}w_{i_n}
	$ and $
	w_{k_0}x_{l_1}w_{k_1}x_{l_2}\cdots w_{k_{m-1}}x_{l_m}w_{k_{m}}
	$ of $\mathcal{B}_{\W}$ is given by
	first juxtaposition 
	$
	w_{i_0}x_{j_1}w_{i_1}x_{j_2}\cdots w_{i_{n-1}}x_{j_n}w_{i_{n}}w_{k_0}x_{l_1}w_{k_1}x_{l_2}\cdots w_{k_{m-1}}x_{l_m}w_{k_{m}}
	$
	and then expanding $w_{i_{n}}w_{k_0}= \sum_{p\in\mathcal I} \alpha_p w_p$, $\alpha_p\in F$.
	So, $\W$ is also understood as some sort of non-commutative polynomials with coefficients in $W$. Clearly, the free $W$-algebra is endowed with a $W$-bimodule action that satisfies relations \eqref{conditions W-algebra} determined by first juxtaposition
	$$
	w_k (w_{i_0}x_{j_1}w_{i_1}x_{j_2}\cdots w_{i_{n-1}}x_{j_n}w_{i_{n}}) w_l= w_k w_{i_0}x_{j_1}w_{i_1}x_{j_2}\cdots w_{i_{n-1}}x_{j_n}w_{i_{n}}w_l,
	$$
	and then expanding $w_k w_{i_0}$ and $w_{i_{n}}w_l$ in the given basis $\mathcal{B}_W$ of $W$, for $w_k, w_l\in \mathcal{B}_W$ and $w_{i_0}x_{j_1}w_{i_1}x_{j_2}\cdots w_{i_{n-1}}x_{j_n}w_{i_{n}}\in \mathcal{B}_{\W}$.
	The elements of the free $W$-algebra are called \emph{generalized $W$-polynomials} or simply \emph{generalized polynomials} when the role of $W$ is clear.
	A \emph{$T_W$-ideal} of the free $W$-algebra is an $W$-ideal which in addition is invariant under all algebra endomorphisms $\varphi$ of $\W$ such that $\varphi(wfv)=w\varphi(f)v$ for all $f \in \W$ and $w,v\in W$; by the universal property, under the endomorphisms that we call \emph{substitutions}, which send variables of $x_i\in X$ in elements of $\W$.
	
	\smallskip
	
	
	Given a $W$-algebra $A$, a generalized polynomial $f(x_1, \dots , x_n)\in \W$ is a \emph{generalized $W$-identity}, or simply \emph{generalized identity} if there is not ambiguity about $W$, of $A$ if $f(a_{1},\dots,a_{n})=0$ for any $a_1,\ldots,a_n\in A$, i.e., $f$ is in the kernel of every homomorphism from $\W$ to $A$. We denote by $\gI_W(A)$, or simply $\gI(A)$ when ambiguity does not arise, the set of differential identities of $A$, which is a $T_W$-ideal of the free $W$-algebra. Remark that in case $W=F$, then we are dealing with the ordinary polynomial identities.

	\smallskip
	
	For $n\geq 1$, we denote by $GP_n^W$, or simply $GP_n$, the vector space of \emph{multilinear generalized polynomials} with coefficient in $W$ in the variables $x_{1},\dots,x_{n}$, so that
	$$
	GP_n:=\spn_F\{w_{i_0}x_{\sigma(1)}w_{i_1} x_{\sigma(2)}\cdots w_{i_{n-1}}x_{\sigma(n)}w_{i_{n}} \ | \ \sigma\in S_{n} ,w_{i_0},\ldots,w_{i_{n}}\in \mathcal{B}_W \},
	$$
	where $S_n$ denotes the symmetric group acting on $\{1,\ldots,n\}$.
	As in the ordinary case, since $F$ has characteristic zero, a Vandermonde argument and the standard linearization procedure show that the $T_W$-ideal $\GId(A)$ is completely determined by its multilinear generalized polynomials (see \cite[Proposition 4.2.3]{Drenskybook}).
	We also consider the vector space
	$$
	GP_n(A):= \dfrac{GP_n}{GP_n \cap \gI(A)},
	$$
	and its dimension $gc_n(A):=\dim_F GP_n(A)$ is the $n$th \emph{generalized codimension} of $A$. Remark that if $W$ is a finite-dimensional algebra, then $gc_n(A)$ is finite for $n \geq 1$.
	
	The symmetric group $S_n$ acts naturally on the left on $GP_n$ by permuting the variables: for $\sigma\in S_{n}$, $\sigma(wx_{i}v)=wx_{\sigma(i)}v$. Since $GP_n\cap \GId(A)$ is stable under this $S_n$-action, the space $GP_n(A)$ is a left $S_{n}$-module and its character, denoted by $g\chi_{n}(A)$, is called the {\em $n$th generalized cocharacter} of $A$. Also, since the characteristic of $F$ is zero,
	$$
	g\chi_n(A)=\sum_{\lambda\vdash n} m_{\lambda}\chi_{\lambda},
	$$
	where $\lambda$ is a partition of $n$, $\chi_{\lambda}$ is the irreducible $S_{n}$-character associated to $\lambda$ and $m_{\lambda}\geq 0$ is the corresponding multiplicity.
	
	
	A variety of $W$-algebras generated by a $W$-algebra $A$ is denoted by $\gv_{W}(A)$, or simply $\gv(A)$, and is called {\em generalized $W$-variety}, or simply {\em generalized variety}, and $\GId(\mathcal{V}):=\GId(A)$. The {\em growth} of $\mathcal{V}= \gv(A)$ is the growth of the sequence $gc_{n}(\mathcal{V}):=gc_{n}(A)$, $n\geq 1$. We say that the generalized variety $\mathcal{V}$ has {\em polynomial growth} if $gc_{n}(\mathcal{V})$ is polynomially bounded and $\mathcal{V}$ has {\em almost polynomial growth} if  $gc_{n}(\mathcal{V})$ is not polynomially bounded but every proper generalized subvariety of $\mathcal{V}$ has polynomial growth.
	
	\smallskip

	In the last part of this section our focus will be on generalized polynomials that are trivial. Recall that a generalized polynomial $f\in \W$ is said {\em $W$-trivial}, or simply {\em trivial}, if $f=0$; otherwise $f$ is {\em $W$-nontrivial}, or simply {\em nontrivial}. Since determining whether a generalized polynomial is trivial is not always straightforward, we shall introduce some techniques and approaches that can help.
	
	Let $f=f(x_1,\dots,x_n)\in GP_n$ be a multilinear generalized $W$-polynomial in the variables $x_1,\dots,x_n$.  Given $\sigma\in S_n$,  we denote by $f_\sigma$ the sum of the
	monomials of $f$ in which the variables occur exactly in the order $x_{\sigma(1)}, x_{\sigma(2)},\dots, x_{\sigma(n)}$, and we call it a {\em generalized monomial} of $f$. If $A$ is a $W$-algebra, then $f$ is called {\em $A$-proper} if $f_\sigma\notin\GId(A)$ for some $\sigma\in S_n$. Clearly, if $f$ is $W$-proper, then it is $W$-nontrivial. In general, the converse is not always true. Although, when $W=F$ ``proper" and ``nontrivial" are synonymous.


	Now we shall focus on linear elements of $W\langle X\rangle$ in a single variable $x$, i.e., elements of $GP_1=\spn_F \{ w x v \, | \, w,v\in \mathcal{B}_W \}.$
	Let us introduce the following notation. Let $\mbox{End}_F( W)$ be the algebra of endomorphism of $A$ with product $\circ$ given by the usual composition of function. Denote by $ L,R : W \longrightarrow \mbox{End}_F( W)$ the operators of left and right multiplications, i.e., 
	for $w \in  W$, the \textit{left} (resp. \textit{right}) \textit{multiplication} by $w$ is the endomorphism $L_w :    W \longrightarrow  W$ (resp. $R_w :   W \longrightarrow W$) of $ W$ defined by 
	\begin{equation*}
		\begin{split}
			L_w(v):=w v \hspace*{1cm}
		\end{split}
		\begin{split}
			\Big(\mbox{resp. }  R_w(v):=v w \Big),
		\end{split}
	\end{equation*}
	for all  $v\in W$, and consider $L_W R_W:=\spn_F \{L_w \circ R_v\ | \ w,v\in \mathcal{B}_W\}\subseteq\mbox{End}_F(W)$.
    
	\begin{lemma}\label{lem: L_WR_W}
		The linear map $\varphi:L_W R_W \longrightarrow GP_1$ given by
		$$
		\varphi(L_w \circ R_v)=w x v,
		$$
		for any $w,v\in \mathcal{B}_W$, is an isomorphism.
	\end{lemma}
	\begin{proof}
		Clearly, $\varphi$ is surjective. Now, let $\sum_{i=1}^m \alpha_i L_{w_i}\circ R_{v_i}\in \ker\varphi$, where $w_i,v_i\in \mathcal{B}_W $ and $\alpha_i\in F$ for $1\leq i \leq m$. Then $f=\sum_{i=1}^m \alpha_i w_ixv_i=0$, i.e., $f$ is a generalized polynomial $W$-trivial. As a consequence, if we consider $W$ as a $W$-algebra with the natural the left and right $W$-actions defined by multiplication, then it follows that $\sum_{i=1}^m \alpha_i w_iav_i=0$ for all $a\in W$, i.e.,  $\sum_{i=1}^m \alpha_i L_{w_i}\circ R_{v_i}=0$. Thus $\varphi$ is also injective, as required.
	\end{proof}
	
	As a direct consequence of Lemma \ref{lem: L_WR_W}, we have the following criterion that establishes whether a linear generalized polynomial in one variable is trivial or not. 
	\begin{proposition}\label{Criterion trivial polynomials}
		Let $f=\sum_{i=1}^m \alpha_i w_ixv_i\in GP_1$. Then $f$ is $W$-trivial if and only if $\sum_{i=1}^m \alpha_i L_{w_i} \circ R_{v_i}=0$.
	\end{proposition}

	\begin{corollary}\label{cor: No GI in 1 var}
		Let $W$ be $W$-algebra with the left and right actions defined by multiplication and $f\in \W$. If $f\in GP_1\cap \GId(W)$, then $f$ is $W$-trivial.
	\end{corollary}
	\begin{proof}
		Let $f=\sum_{i=1}^m \alpha_i w_ixv_i\in GP_1\cap \GId(W)$, where $w_i,v_i\in \mathcal{B}_W $ and $\alpha_i\in F$ for $1\leq i \leq m$. If $\alpha_i=0$ for all $1\leq i \leq m$, then $f$ is $W$-trivial. So, let us assume that $\alpha_i\neq 0$ for some $1\leq i\leq m$. Since $f\in \GId(W)$, $\sum_{i=1}^m \alpha_i w_i a v_i=0$ for all $a\in W$, i.e., $\sum_{i=1}^m \alpha_i L_{w_i} \big(R_{v_i}(a)\big)=0$ for all $a\in W$, and by Proposition \ref{Criterion trivial polynomials} $f$ is $W$-trivial.
	\end{proof}
	
	So, when we consider $W$ as $W$-algebra with the natural left and right $W$-actions defined by multiplication, then there are no nonzero linear generalized identities in one variable. It is important to notice that in case we are considering $W$ with the structure of $W$-algebra given by another action, then this result is not in general true (see Section \ref{sezionecinque}).

	\smallskip
	
	In what follows 
	we shall assume that $W=UT_2$, the algebra of $2 \times 2$ upper triangular matrices over $F$, i.e., we shall work in the class of $UT_2$-algebras. Also, we shall consider as a basis the set $\mathcal{B}_{UT_2}=\{1:=e_{11}+e_{22}, e_{22},e_{12}\}$, where $e_{ij}$'s are the standard matrix units.

	\section{Generalized polynomial identities of $UT_2$}\label{sezionetre}

	In this section we shall compute a basis for the $T$-ideal of generalized identities, and the corresponding codimension sequence, of $UT_2$ as $UT_2$-algebra with the left and right $UT_2$-actions given by the usual multiplication.
	
	Let $[x_1,x_2]:=x_1x_2-x_2x_1$ be the {\em commutator} of $x_1$ and $x_2$. Also, in what follows we use $[x_1,x_2,\ldots, x_k]$ to denote a left normed commutator. A straightforward computation shows that the following polynomial is a generalized polynomial identity of $UT_2:$
	\begin{equation}\label{Id UT_2}
		[x_1,x_2]- [x_1,x_2,e_{22}]\equiv 0.
	\end{equation}
	Also, it is a $UT_2$-nontrivial polynomial since it is $UT_2$-proper. Next, we find some consequences that we will use to reach our goal.
	
	\begin{lemma}\label{conseguenze}
		The following polynomials are generalized identities of $UT_2$ and consequences of  \eqref{Id UT_2}:
		\begin{equation*}
			e_{22}[x_1,x_2]; \ [x_1,x_2]- [x_1,x_2]e_{22}; \ [x_1,x_2][x_3,x_4]; \   [x_1,x_2]e_{12}; \ e_{12}[x_1,x_2].
		\end{equation*}
	\end{lemma}
	\begin{proof}
		By acting on \eqref{Id UT_2} by $e_{22}$ from the right we get that $e_{22}[x_1,x_2]e_{22}\equiv 0$. By Proposition \ref{Criterion trivial polynomials} $e_{22}xe_{22}=e_{22}x$, then it follows that $e_{22}[x_1,x_2]\equiv 0$.
		Moreover, as a consequence of \eqref{Id UT_2} and $e_{22}[x_1,x_2]\equiv 0$, we obtain that $[x_1,x_2]- [x_1,x_2]e_{22}\equiv 0.$
		
		Also, by multiplying $[x_1,x_2]- [x_1,x_2]e_{22}\equiv 0$  by $[x_3,x_4]$ on the right and by using $e_{22}[x_1,x_2]\equiv 0$ we get $[x_1,x_2][x_3,x_4]\equiv 0.$

		Finally, the generalized identities $e_{12} [x_1,x_2]\equiv 0$ and $[x_1,x_2]e_{12}\equiv 0$ follow from $e_{22}[x_1,x_2]\equiv 0$ and $[x_1,x_2]- [x_1,x_2]e_{22}\equiv 0$, respectively, by acting by $e_{12}$ respectively from the left and the right.
	\end{proof}
	
	We are now in position to prove that the generalized polynomial \eqref{Id UT_2} span $\GId(UT_2)$ as $T$-ideal. 
	
	\begin{theorem}\label{basedelTideale}
		Let $UT_2$ be the $UT_2$-algebra with the action given by the right and the left multiplication. Then $\GId(UT_2)$ is generated, as $T_{\UT}$-ideal, by the following polynomial:
		$$
		[x_1,x_2]- [x_1,x_2,e_{22}].
		$$
		Moreover, $gc_n(UT_2) = (n+2)2^{n-1}+2.$
	\end{theorem}
	\begin{proof}
		Let $I$ be the $T_{UT_2}$-ideal generated by the above polynomials. It is clear that $I\subseteq\GId(UT_2).$ In order to prove the opposite inclusion, let  $w$ be a monomial of $GP_n.$
		If $w$ does not contain any $e_{22}$ or $e_{12}$, i.e., it is an ordinary monomial, then, since $[x_1,x_2][x_3,x_4]\in I$ (Lemma \ref{conseguenze}) and by applying the well-known reduction process modulo the ordinary polynomial identities of $UT_2$ (see for instance \cite[Theorem 4.1.5]{GiambrunoZaicevbook}), $w$ can be written as a linear combination of $x_1 x_2\cdots x_n$ and
		$$
		x_{l_1}\cdots x_{l_m}[x_k,x_{p_1},\ldots, x_{p_{n-m-1}}],
		$$
		where $0\leq m\leq n-2,$ $l_1<\cdots < l_m$ and $k>p_1<\cdots <p_{n-m-1}.$ 
		
		Now, suppose that in $w$ appears at least one $e_{22}$. By Proposition \ref{Criterion trivial polynomials}, $e_{22}xe_{22}=e_{22}x,$ $e_{22}xe_{12}=0$ and $e_{12}xe_{22}=e_{12}x$, then we may assume that $w$ contains exactly one $e_{22}.$ By the  Poincarè-Birkhoff-Witt Theorem and by $e_{22}[x_1,x_2]\in I$ (Lemma \ref{conseguenze}), $w$ can be written as a linear combination of $e_{22}x_1x_2\cdots x_n$ and polynomials of the type
		$$
		x_{i_1}\cdots x_{i_r}c_1\cdots c_s,
		$$
		where $i_1<\cdots <i_r$ and $c_1,\ldots, c_s$ are left-normed commutators and just one of them contains $e_{22}.$ Since $[x_1,x_2][x_3,x_4]\in I$ (Lemma \ref{conseguenze}), then $s=2$ and one in between of the two commutators $c_1,c_2$ contains $e_{22}$. Now, since $e_{22}[x_1,x_2], \, [x_1,x_2]-[x_1,x_2]e_{22}\in I$ (Lemma \ref{conseguenze}), we may assume that $s=1$.
		Moreover, since $[x_1,x_2]-[x_1,x_2]e_{22}\in I$ (Lemma \ref{conseguenze}), then we can assume that $e_{22}$ appears in the second position of the commutator (otherwise we can erase $e_{22}$ and come back to the previous case of ordinary polynomials). Now, take the left-normed commutator $[x_k,e_{22},x_{j_1},\ldots, x_{j_s}]$ and notice that using the same reasoning as we did before, we may assume that $j_1<\cdots < j_s.$ Also, by the Jacobi identity $[x_2, e_{22},x_1] = [x_1, e_{22},x_2]-[x_1,x_2, e_{22}]$ it turns out that
		$$
		[x_2, e_{22},x_1]\equiv [x_1,e_{22},x_2] - [x_1,x_2] \ \text{(mod $I$)}.
		$$
		This implies that the left-normed commutator can be written as $[x_{i_1},e_{22},x_{i_2},\ldots, x_{i_s}]$ where $i_1<i_2<\cdots < i_s.$
		
		Finally, let $w$ be a monomial of $GP_n$ containing at least one $e_{12}.$  Again by Proposition \ref{Criterion trivial polynomials},
		$e_{12}xe_{12}= 0,$ $e_{22}xe_{12}=0$ and $e_{12}xe_{22}=e_{12}x$, then $w$ must contain just one $e_{12}.$ Moreover, since by Lemma \ref{conseguenze} $e_{12}[x_1,x_2], \,[x_1,x_2]e_{12}\in I,$ all the variables on the left and on the right of $e_{12}$ are ordered. Thus $w$ can be written modulo $I$ as
		$$
		x_{i_1}\cdots x_{i_r}e_{12}x_{j_1}\cdots x_{j_{n-r}},
		$$
		where $0\leq r\leq n,$ $i_1<\cdots <i_r$ and $j_1<\cdots <j_{n-r}.$

		By putting together all the previous remarks, we have proved that $GP_n$ is generated modulo $I$ by the polynomials:
		\begin{equation}\label{base non identita}
			\begin{split}
				&x_1\cdots x_n; \\
				&e_{22}x_1\cdots x_n; \\
				&X_{12}^{(\mathcal{I})} = x_{i_1}\cdots x_{i_r}e_{12}x_{j_1}\cdots x_{j_{n-r}}; \\
				&X^{(\mathcal{L},k)} = x_{l_1}\cdots x_{l_s}[x_k,x_{m_1},\ldots, x_{m_t}]; \\
				&X_{22}^{(\mathcal{P})} = x_{p_1}\cdots x_{p_u}[x_{q_1}, e_{22},x_{q_2},\ldots, x_{q_v}]
			\end{split}
		\end{equation}
		where $\mathcal{I}=\{i_1,\ldots, i_r\},$ $\mathcal{L}=\{l_1,\ldots, l_s\}$ and $\mathcal{P}=\{p_1,\ldots, p_u\}$ are subsets of $\{1,\ldots, n\},$ $i_1<\cdots < i_r,$ $j_1<\cdots <j_{n-r},$ $l_1<\cdots <l_s,$ $k>m_1<\cdots <m_t,$ $p_1<\cdots <p_u,$ $q_1<q_2<\cdots <q_v,$ $0\leq r\leq n,$ $t\geq 1$ and $v\geq 1.$
		\black
		
		Next we prove that these elements are linearly independent modulo $\GId(UT_2).$ To this end, let 
		$$
		f =\alpha_1x_1\cdots x_n+ \alpha_2e_{22}x_1\cdots x_n+\sum_{\mathcal{I}}\beta_{\mathcal{I}} X_{12}^{(\mathcal{I})}+\sum_{\mathcal{L},k}\gamma_{\mathcal{L},k}X^{(\mathcal{L},k)}+\sum_{\mathcal{P}}\delta_{\mathcal{P}} X_{22}^{(\mathcal{P})}
		$$
		be a linear combination of the generalized polynomials \eqref{base non identita} and suppose by contradiction that $f\neq 0.$ We shall make suitable evaluations to prove that $f=0$ and this will complete the proof.
		
		First, if we evaluate $x_1=\cdots = x_n= e_{11},$ then we get $\alpha_1e_{11}+\beta_{\mathcal{I}}e_{12}=0,$ where $\mathcal{I}=\{1,\ldots, n\},$ thus $\alpha_1 =  \beta_{\mathcal{I}}= 0.$ Now let us make the evaluation $x_1=\cdots = x_n = e_{22}.$ In this case we get $\alpha_2e_{22}+\beta_{\mathcal{I}'}e_{12}=0,$ where $\mathcal{I}'=\emptyset,$ so $\alpha_2 = \beta_{\mathcal{I}'}=0.$ For a fixed $\mathcal{I}=\{i_1,\ldots, i_r\},$ we set $x_{i_1}=\cdots = x_{i_r}= e_{11}$ and $x_{j_1}=\cdots = x_{j_{n-r}}= e_{22}$ and we get $\beta_{\mathcal{I}}e_{12}=0$ thus $\beta_{\mathcal{I}}=0.$ Now, for fixed $\mathcal{L}=\{l_1,\ldots, l_s\}$ and $k,$ from the evaluation $x_{l_1}=\cdots = x_{l_s}=e_{11}+e_{22},$ $x_k= e_{12}$ and $x_{m_1}= \cdots = x_{m_t}= e_{22},$ we get $\gamma_{\mathcal{L},k}e_{12}=0,$ thus $\gamma_{\mathcal{L},k}=0.$ Here remark that all the polynomials of the type $X_{22}^{(\mathcal{P})}$ evaluate to zero since in $X^{(\mathcal{L},k)}$ it must be $k>m_1<\cdots < m_t$ whereas in $X_{22}^{(\mathcal{P})}$ it must be $q_1<q_2<\cdots <q_v.$ Finally, for any fixed $\mathcal{P}=\{p_1,\ldots, p_u\},$ we make the substitution $x_{p_1}=\cdots = x_{p_u}= e_{11}+e_{22},$ $x_{q_1}=e_{12}$ and $x_{q_2}= \cdots = x_{q_v}= e_{22}$ and we get $\delta_{\mathcal{P}}e_{12}=0,$ that is $\delta_{\mathcal{P}}=0.$
		Therefore, all the scalars appearing in $f$ are zero, i.e., $f=0,$ a contradiction.
		
		Thus the elements in \eqref{base non identita} are linearly independent modulo $\GId(UT_2)$ and, since $GP_n\cap I\subseteq GP_n\cap \GId(UT_2),$ this proves that $\GId(UT_2)= I$ and the polynomials in \eqref{base non identita} are a basis of $GP_n$ modulo $GP_n\cap \GId(UT_2).$ Hence, by counting we get
		\begin{align*}
			gc_n(UT_2)& =  2+\sum_{r=0}^n\binom{n}{r}+\sum_{r=2}^n\binom{n}{r}(r-1)+\sum_{r=0}^{n-1}\binom{n}{r} = 2+\sum_{r=0}^n\binom{n}{r} +\sum_{r=1}^n r\binom{n}{r}-1-\sum_{r=0}^n\binom{n}{r}+2+\sum_{r=0}^{n}\binom{n}{r}-1\\
			&= (n+2)2^{n-1}+2.
		\end{align*}
	\end{proof}

	\section{Generalized cocharacter sequence of $UT_2$}\label{sezionequattro}
	
	In this section, we shall determine the generalized cocharacter of $UT_2$ as $UT_2$-algebra where the action of $UT_2$ as bimodule over itself is the usual product of $UT_2$.
	
	We shall start by proving some technical lemmas that give us a lower bound for the multiplicities $m_\lambda$ of $n$th generalized $UT_2$-cocharacter of $UT_2$
	\begin{equation}\label{eq: generalized cocharacter UT_2}
		g\chi_n(UT_2)= \sum_{\lambda\vdash n} m_\lambda \chi_\lambda.
	\end{equation}
	To this end recall that any irreducible left $S_n$-module $W_{\lambda}\subseteq GP_n$ with character $\chi_\lambda$ can be generated as $S_n$-module by an element of the form $e_{T_\lambda}f$, for some $f\in W_\lambda$ and
	some tableau $T_\lambda$ of shape $\lambda$. Here $e_{T_{\lambda}}=\sum_{\substack{\sigma\in R_{T_{\lambda}} \\ \tau\in C_{T_{\lambda}}}} (\sgn\tau)\sigma\tau$ is a minimal quasi-idempotent corresponding to $T_\lambda$, where $R_{T_{\lambda}}$ and $C_{T_{\lambda}}$ are the subgroups of $S_n$ stabilizing the rows and columns of $T_{\lambda}$, respectively.

	\begin{lemma}\label{lem: lower multiplicities row}
		$m_{(n)}\geq 2n+3$ in \eqref{eq: generalized cocharacter UT_2}.
	\end{lemma}
	\begin{proof}
		Let us consider the standard tableau
		\begin{equation*}
			T_{(n)}=\begin{array}{|c|c|c|c|}\hline
				1 & 2 & \cdots &n \\ \hline
			\end{array}\;,
		\end{equation*}
		and the following $2n+3$ generalized polynomials associated to it
		\begin{align}
			& a(x)=x^n  ,\label{eq: hwv a11}\\
			& a_{22}^{(0)}(x)= e_{22} x^n ,\label{eq: hwv a22 0}\\
			& a_{22}^{(i)}(x)=x^{i-1} [x, e_{22}] x^{n-i} , \quad 1\leq i \leq n\label{eq: hwv a22},\\
			& a_{12}^{(j)}(x)=x^j e_{12} x^{n-j}, \quad 0\leq j \leq n. \label{eq: hwv a12}
		\end{align}
		These polynomials are obtained from the quasi-idempotents corresponding to the tableau $T_{(n)}$ by identifying all the elements. 
		Clearly, the polynomials \eqref{eq: hwv a11}--\eqref{eq: hwv a12} do not vanish in $UT_{2}$. We claim that these generalized polynomials are linear independent modulo $\gI(UT_2)$. So, let $f\in \gI(\UT)$ be a linear combination of such polynomials, i.e.,
		$$
		f=\alpha a(x) + \sum_{i=0}^{n} \beta_i a_{22}^{(i)}(x) + \sum_{j=0}^{n} \gamma_j a_{12}^{(j)}(x).
		$$
		First suppose that $\alpha \neq 0$ or $ \gamma_n\neq 0$. Then, by making the evaluation $x=e_{11}$ one gets $\alpha e_{11}+ \gamma_n e_{12}=0$. Hence, it follows that $\alpha=\gamma_n=0$, a contradiction. 
		
		Now assume that $\beta_0\neq 0$ or $ \gamma_0\neq 0$.  Then, if we consider the evaluation $x= e_{22}$, we obtain  $\beta_0 e_{22}+ \gamma_0 e_{12}=0$. Hence, it follows that $\beta_0=\gamma_0=0$, a contradiction.
		
		Next, assume that there exists $\gamma_j\neq 0$ for some $1\leq j \leq n-1$. If we substitute $x=\delta e_{11}+e_{22}$ with $\delta\in F$, $\delta\neq 0$, we get $ \sum_{j=1}^{n} \delta^{j}\gamma_j=0$. Since $F$ is an infinite filed, we can choose $\delta_{1},\dots,\delta_{n-1}\in F$ such that $\delta_{i}\neq 0$ and $\delta_{i}\neq\delta_{j}$ for $1\leq i\neq j\leq n-1$. Then we get the following homogeneous linear system of $n-1$ equations in the $n-1$ variables $\gamma_1, \dots, \gamma_{n-1}$
		\begin{equation*}
			\label{eq: sistema 1 riga 1}
			\sum_{j=1}^{n-1}\delta_{k}^{j}\gamma_j=0,\quad 1\leq k \leq n-1.
		\end{equation*}
		Since the matrix associated to the above system is a Vandermonde matrix, it follows that $\gamma_j=0$, for any $1\leq j \leq n-1$, a contradiction.
		
		Finally, if $\beta_i\neq 0$ for some $1\leq i \leq n$, then by making the substitution $x=\delta e_{11} + e_{22} + e_{12}$, we get that $\sum_{i=1}^{n} \delta^i \beta_i=0$. Now, as above, one may choose distinct $\delta_1, \dots ,\delta_n \in F$ such that $\delta_i\neq 0$ for $1\leq i \leq n$. Hence, we obtain the following linear system of $n$ equations in the $n$ variables $\beta_1, \dots, \beta_n$
		\begin{equation*}
			\sum_{i=1}^{n}\delta_{k}^{i}\beta_i=0,\quad 1\leq k \leq n.
		\end{equation*}
		Again, we obtained a linear system whose associated matrix is a Vandermonde matrix. Thus, it follows that $\beta_i=0$ for any $1\leq i \leq n$, a contradiction. 
		Therefore the $2n+3$ generalized polynomials \eqref{eq: hwv a11}--\eqref{eq: hwv a12} are linearly independent modulo $\GId(\UT)$. 
		
		Notice that the complete linearization of $a(x)$ is $e_{{(n)}}(x_{1},\dots,x_{n})=e_{T_{(n)}}(x_{1}\cdots x_{n}),$ and, for every $0\leq i \leq n$, the complete linearization of $a_{22}^{(i)}(x)$ and $a_{12}^{(i)}(x)$ are the polynomials
		$e_{{(n)}}^{e_{22},i}(x_{1},\dots,x_{n})=e_{T_{(n)}}(x_{1}\cdots x_{i-1} [x_i, e_{22}] x_{i+1}\cdots x_{n})$ and 
		$e_{{(n)}}^{e_{12},i}(x_{1},\dots,x_{n})=e_{T_{(n)}}(x_{1}\cdots x_i e_{12}x_{i+1}\cdots x_{n}),$
		respectively. Then it follows that the polynomials $e_{{(n)}}$, $e_{{(n)}}^{e_{22},i}$, $e_{{(n)}}^{e_{12},i}$ are also linearly independent modulo $ \Id^{\varepsilon}(UT_{2})$ and as a consequence $m_{(n)}\geq 2n+3$, as desired.
	\end{proof}
	
	\begin{lemma}
		\label{lem: lower multiplicities 2 rows}
		If $p\geq 1$ and $q\geq 0$, then $m_{(p+q,p)} \geq 3(q+1)$ in \eqref{eq: generalized cocharacter UT_2}.
	\end{lemma}
	\begin{proof}
		For any $0\leq i \leq q$, let $T_{(p+q,p)}^{(i)}$ be the standard tableau
		\begin{small}
			\begin{equation*}
				\begin{array}{|c|c|c|c|c|c|c|c|c|c|c|}\hline
					i+1 & i+2 & \cdots & i+p-1 & i+p & 1 & \cdots & i & i+2p+1 & \cdots & n \\ \hline
					i+p+2 & i+p+3 & \cdots & i+2p & i+p+1\\ \cline{1-5}
				\end{array}\:,
			\end{equation*}
		\end{small}
		and let associate to it the following generalized polynomials
		\begin{align*}
			&b^{(i)}_{p,q}(x,y)=x^{i}\underbrace{\bar{x}\cdots\tilde{x}}_{p-1}[x,y]\underbrace{\bar{y}\cdots\tilde{y}}_{p-1}x^{q-i},\\
			&c^{(i)}_{p,q}(x,y)=x^{i}\underbrace{\bar{x}\cdots\tilde{x}}_{p-1}(xe_{12}y-ye_{12}x)\underbrace{\bar{y}\cdots\tilde{y}}_{p-1}x^{q-i},\\
			&d^{(i)}_{p,q}(x,y)=x^{i}\underbrace{\bar{x}\cdots\tilde{x}}_{p-1}(xe_{22}y-ye_{22}x)\underbrace{\bar{y}\cdots\tilde{y}}_{p-1}x^{q-i},
		\end{align*}
		where the symbol $\bar{}$ or $\tilde{}$ means alternation on the corresponding variables. For any $1\leq i \leq q$, these polynomials are obtained from the quasi-idempotents corresponding to the tableau $T_{(p+q,p)}^{(i)}$ by identifying all the elements in each row. Also, they are not generalized identities of $\UT$. 
		
		Next, we shall show that the generalized polynomials $b^{(i)}_{p,q}(x,y)$, $c^{(i)}_{p,q}(x,y)$ $d^{(i)}_{p,q}(x,y)$, $0\leq i \leq q$, are linear independent modulo $\GId(\UT)$. To this end, let us consider
		$$
		f= \sum_{i=0}^{q} \alpha_i b^{(i)}_{p,q}(x,y) + \sum_{i=0}^{q} \beta_i c^{(i)}_{p,q}(x,y) + \sum_{i=0}^{q} \gamma_i d^{(i)}_{p,q}(x,y)\in \GId(\UT).
		$$
		First, suppose that there exists $\beta_i\neq 0$ for some $0\leq i \leq q$. By evaluating $x=\delta e_{11}+e_{22}$, with $\delta\in F$, $\delta\neq 0$, and $y=e_{11}$, we get $\sum_{i=0}^{q}(-1)^{p-1}\delta^{i} \beta_i=0.$
		Since $F$ is infinite, we may take $\delta_{1},\dots,\delta_{q+1}\in F$, where $\delta_{j}\neq 0$, $\delta_{j}\neq\delta_{k}$, for all $1\leq j\neq k\leq q+1$. Thus, as in the proof of the previous lemma, we obtain the following homogeneous linear system of $q+1$ equations in the $q+1$ variables $\beta_0, \dots, \beta_{q}$
		\begin{equation}
			\label{SistemaTabella2riga}
			\sum_{i=0}^{q}\delta_{j}^{i}\beta_{i}=0,\quad 1\leq j \leq q+1.
		\end{equation}
		Since the matrix associated to the system above is a Vandermonde matrix, it follows that $\beta_i=0$, for all $0\leq i \leq q$. 
		
		Now if there exists $\alpha_i\neq 0$ for some $0\leq i \leq q$, then we substitute  $x=\delta e_{11}+ e_{12}+e_{22}$, with $\delta\in F$, $\delta\neq 0$, and $y=e_{11}$, and we get $\sum_{i=0}^{q}(-1)^{p-1}\delta^{i} \alpha_i=0.$ Thus, as above, since $F$ is infinite, we obtain a homogeneous linear system of $q+1$ equations in the $q+1$ variables $\alpha_0,\dots, \alpha_q$ equivalent to \eqref{SistemaTabella2riga}.  Therefore $\alpha_{i}=0$ for all $0\leq i \leq q$, a contradiction.
		
		Finally, assume that there exists $\gamma_i\neq 0$ for some $0\leq i \leq q$. By making the evaluation  $x=\delta e_{11}+ e_{12}+e_{22}$, with $\delta\in F$, $\delta\neq 0$, and $y=e_{22}$, we obtain $\sum_{i=0}^{q}\delta^{i} \gamma_i=0.$ Then, as above, we get a homogeneous linear system of $q+1$ equations in the $q+1$ variables $\gamma_0, \dots, \gamma_q$ equivalent to \eqref{SistemaTabella2riga}.  So, $\gamma_i=0$ for all $0\leq i \leq q$, a contradiction.
		
		Thus, the $3(q+1)$ generalized polynomials $b^{(i)}_{p,q}(x,y)$, $c^{(i)}_{p,q}(x,y)$ $d^{(i)}_{p,q}(x,y)$, $0\leq i \leq q$, are linearly independent modulo $\GId(\UT)$ and, so, as in Lemma \ref{lem: lower multiplicities row}, $m_{(p+q,p)}\geq 3(q+1)$, as required.
	\end{proof}

	\begin{lemma}
		\label{lem: lower multiplicities 3 rows}
		If $p\geq 1$ and $q\geq 0$, then $m_{(p+q,p,1)} \geq q+1$ in \eqref{eq: generalized cocharacter UT_2}.
	\end{lemma}
	\begin{proof}
		For any $0\leq i \leq q$, define $T_{(p+q,p,1)}^{(i)}$ to be the standard tableau
		\begin{small}
			\begin{equation*}
				\begin{array}{|c|c|c|c|c|c|c|c|c|c|}\hline
					i+p & i+1 & \cdots &  i+p & 1 & \cdots & i & i+2p+2 & \cdots & n \\ \hline
					i+p+1 & i+p+3 & \cdots & i+2p+1 \\ \cline{1-4}
					i+p+2\\
					\cline{1-1}
				\end{array}\:,
			\end{equation*}
		\end{small}
		and associate to it the generalized polynomial
		\begin{equation*}
			h^{(i)}_{p,q}(x,y,z)=x^{i}\underbrace{\hat{x}\cdots\tilde{x}}_{p-1}\bar{x} \bar{y}\bar{z}\underbrace{\hat{y}\cdots\tilde{y}}_{p-1}x^{q-i}, 
		\end{equation*}
		where the symbol $\hat{}$ or $\tilde{}$ or $\bar{}$ means alternation on the corresponding variables.  For any $1\leq i \leq q$ these generalized polynomials are obtained from the quasi-idempotents corresponding to the tableau $T_{(p+q,p,1)}^{(i)}$ by identifying all the elements in each row. Clearly, $h^{(i)}_{p,q}(x,y,z)$, $1\leq i \leq q$,  do not belong to $\GId(\UT)$. We claim that the $q+1$ generalized polynomials $h^{(i)}_{p,q}(x,y,z)$, $0\leq i \leq q$, are linear independent modulo $\GId(\UT)$. If not, there exist $\alpha_0, \dots, \alpha_q\in F$ not all zero such that
		$$
		\sum_{i=0}^{q} \alpha_i h^{(i)}_{p,q}(x,y,z) \in \GId(\UT).
		$$
		If we substitute $x=\beta e_{11}+ e_{12}+e_{22}$, with $\beta\in F$, $\beta\neq 0$, $y=e_{11}$, and $z=e_{22}$, then we obtain $\sum_{i=0}^{q} \beta^i \alpha_i=0$, and again with a Vandermonde argument we get that  $\alpha_i=0$ for all $0\leq i \leq q$, a contradiction.
		
		Therefore the $q+1$ generalized polynomials $h^{(i)}_{p,q}(x,y,z)$, $0\leq i \leq q$, are linearly independent modulo $\GId(\UT)$, as claimed. Again, as in Lemma \ref{lem: lower multiplicities row}, this implies that $m_{(p+q,p,1)} \geq q+1$.
	\end{proof}

	Next, we shall prove the main theorem of the section.

	\begin{theorem}
		\label{Thm: gcocharacter UT_2}
		If $g\chi_n(\UT)=\sum_{\lambda\vdash n} m_{\lambda} \chi_{\lambda}$ is the $n$th generalized cocharacter of $\UT$, then
		\begin{equation*}
			m_{\lambda} =\begin{cases}
				2n+3, & \mbox{ if } \lambda=(n) \\
				3(q+1), & \mbox{ if } \lambda=(p+q,p) \\
				q+1, & \mbox{ if } \lambda=(p+q,p,1) \\
				0, & \mbox{ in all other cases}
			\end{cases}.
		\end{equation*}
	\end{theorem}
	\begin{proof}
		Let $d_\lambda=\deg \chi_\lambda$  be the degree of $\chi_\lambda$, $\lambda\vdash n$. Then $gc_n(\UT)= \sum_{\lambda\vdash n} m_\lambda d_\lambda$, and by Lemmas \ref{lem: lower multiplicities row}, \ref{lem: lower multiplicities 2 rows} and \ref{lem: lower multiplicities 3 rows} we have that
		\begin{equation} \label{eq: diseq deg cochar}
			gc_n(\UT)\geq  (2n+3) d_{(n)} + \sum_{ \substack{ 1 \leq p \leq \lfloor\frac{n}{2}\rfloor \\ 0 \leq q \leq n-2p}} 3 (q+1) d_{(p+q,p)} + \sum_{ \substack{ 1 \leq p \leq \lfloor\frac{n-1}{2}\rfloor \\ 0 \leq q \leq n-2p-1}}  (q+1) d_{(p+q,p,1)}.
		\end{equation}
		Thus, to complete the proof is enough to show the \eqref{eq: diseq deg cochar} is actually an equality. To this end, notice that for $n=2p+q$, by the hook formula (see for example \cite[Theorem 2.3.21]{JamesKerber1981}) we have that
		$$
		d_{(p+q,p)}= \dfrac{n!}{p!q!(p+q+1)\cdots (q+2)}=\binom{n}{p} \dfrac{n-2p+1}{n-p+1}.
		$$
		Then, it follows that
		\begin{align*}
			\sum_{ \substack{ 1 \leq p \leq \lfloor\frac{n}{2}\rfloor \\ 0 \leq q \leq n-2p}}  (q+1) d_{(p+q,p)} &=(n+1)\sum_{p=1}^{\lfloor\frac{n}{2}\rfloor}\binom{n}{p}-3\sum_{p=1}^{\lfloor\frac{n}{2}\rfloor}\binom{n}{p}p+\sum_{p=1}^{\lfloor\frac{n}{2}\rfloor}\binom{n}{p}\dfrac{p^{2}}{n-p+1}\\
			&=(n+1) \left( \sum_{p=1}^{\lfloor\frac{n}{2}\rfloor}\binom{n}{p}+\sum_{p=n-\lfloor\frac{n}{2}\rfloor+1}^{n}\binom{n}{p} \right) - \sum_{p=1}^{\lfloor\frac{n}{2}\rfloor}\binom{n}{p}p -\sum_{p=n-\lfloor\frac{n}{2}\rfloor+1}^{n}\binom{n}{p}p -2\sum_{p=1}^{\lfloor\frac{n}{2}\rfloor}\binom{n}{p}p  ,
		\end{align*}
		where in the last equality we use that $\binom{n}{n-p+1}=\binom{n}{p}\frac{p}{n-p+1}$. Recall that $i\binom{i-1}{j-1}=j\binom{i}{j}$ and $\sum_{j=0}^{i}\binom{i}{j}=2^{i}$. Hence, if $n=2k$,
		\begin{align*}
			\sum_{ \substack{ 1 \leq p \leq k \\ 0 \leq q \leq n-2p}}  (q+1) d_{(p+q,p)}&=(2k+1)(2^{2k}-1)-k2^{2k}-4k\sum_{p=1}^{k} \binom{2k-1}{p-1}=2^{2k}-2k-1.
		\end{align*}
		In case $n=2k+1$,
		\begin{align*}
			\sum_{ \substack{ 1 \leq p \leq k \\ 0 \leq q \leq n-2p}}  (q+1) d_{(p+q,p)}=&(2k+2) \left(2^{2k+1}-1-\binom{2k+1}{k+1} \right)-(2k+1)2^{2k}+(k+1)\binom{2k+1}{k+1}-2(2k+1)\sum_{p=1}^{k} \binom{2k}{p-1}\\
			=& 2^{2k+1}-2k-2.
		\end{align*}
		Thus, we have that
		\begin{equation}\label{eq: comp car 2 row}
			\sum_{ \substack{ 1 \leq p \leq \lfloor\frac{n}{2}\rfloor \\ 0 \leq q \leq n-2p}}  (q+1) d_{(p+q,p)}=2^{n}-n-1.
		\end{equation}
		Now, for  $n=2p+q+1$, by applying the hook formula again, we get that
		$$
		d_{(p+q,p,1)}= \dfrac{n!}{(p-1)!q!(p+1)(p+q+2)(p+q)\cdots (q+2)}=\binom{n}{p+1} \dfrac{p(n-2p)}{n-p+1}.
		$$
		Then, by recalling that $\binom{n}{p+1}=\binom{n}{p}\frac{n-p}{p+1}$, $\binom{n}{n-p+1}= \binom{n}{p}\frac{n}{n-p+1}$ and $\binom{n}{p+1}+\binom{n}{p}=\binom{n+1}{p+1}$,  it follows
		\begin{align*}
			\sum_{ \substack{ 1 \leq p \leq \lfloor\frac{n-1}{2}\rfloor \\ 0 \leq q \leq n-2p-1}}  (q+1) d_{(p+q,p,1)}= \sum_{p=1}^{ \lfloor\frac{n-1}{2}\rfloor} \left(  \binom{n+1}{p+1}p - \binom{n}{p-1}\right) (n-2p) .
		\end{align*}
		Hence, with a similar computation as above, we obtain that
		\begin{align}\label{eq: comp car 3 row}
			\sum_{ \substack{ 1 \leq p \leq \lfloor\frac{n-1}{2}\rfloor \\ 0 \leq q \leq n-2p-1}}  (q+1) d_{(p+q,p,1)} =  (n-4)2^{n-1}+n+2.
		\end{align}
		Thus, by \eqref{eq: comp car 2 row}, \eqref{eq: comp car 3 row} and since $d_{(n)}=1$, it follows that 
		\begin{align*}
			&(2n+3) d_{(n)} +3 \sum_{ \substack{ 1 \leq p \leq \lfloor\frac{n}{2}\rfloor \\ 0 \leq q \leq n-2p}}  (q+1) d_{(p+q,p)} + \sum_{ \substack{ 1 \leq p \leq \lfloor\frac{n-1}{2}\rfloor \\ 0 \leq q \leq n-2p-1}}  (q+1) d_{(p+q,p,1)}\\
			& \ = 2n+3 + 3(2^{n}-n-1) + (n-4)2^{n-1}+n+2   =(n+2)2^{n-1} +2.
		\end{align*}
		Since by Theorem \ref{basedelTideale} $gc_n(\UT)=(n+2)2^{n-1} +2$, we get that \ref{eq: diseq deg cochar} is actually an equality and we are done.
	\end{proof}

	\section{On almost polynomial growth}\label{sezionecinque}

	In this section, we shall construct a variety of $UT_2$-algebras inside $\gv(UT_2)$ of almost polynomial growth of the codimensions.
	We shall also present another variety of $UT_2$-algebras that have almost polynomial growth of the codimensions but it is not contained in $\gv(UT_2)$.

	Let us define on $UT_2$ a new structure as $UT_2$-bimodule in the following way: let $1:=e_{11}+e_{22}$ and $e_{22}$ act by left and right multiplication as in the previous case and let 
	$$
	e_{12}\cdot a = a\cdot e_{12}=0
	$$ 
	for all $a\in UT_2.$ It readily follows that this action defines $UT_2$ as a new $UT_2$-algebras that we will denote it by $UT_2^D.$ Such a notation is justified by noticing that if we let $D$ be the subalgebra of $UT_2$ spanned by $e_{11}$ and $e_{22},$ then the above action is the natural generalization of the left and right multiplication of $UT_2$ by elements of $D$, i.e., we can also view $UT_2$ as a $D$-algebra.
	
	Following step-by-step the lines of Theorem \ref{basedelTideale} and Theorem \ref{Thm: gcocharacter UT_2} with the necessary changes, we can prove the following results.
	
	\begin{theorem}\label{base del Tideale D algebra}
		Let $UT_2^D$ be the $UT_2$-algebra with the above action. Then $\GId(UT_2^D)$ is generated, as $T_{UT_2}$-ideal,  by the following polynomials:
		$$
		e_{12}x; \quad xe_{12}; \quad [x_1,x_2]- [x_1,x_2,e_{22}].
		$$
		Moreover, $gc_n(UT_2^D) = n2^{n-1}+2.$
	\end{theorem}
	
	\begin{theorem}\label{cocarattere di UT2D}
		Let $g\chi_n(\UT^D)=\sum_{\lambda\vdash n} m_{\lambda} \chi_{\lambda}$ be the $n$th generalized cocharacter of $\UT^D.$ Then
		\begin{equation*}
			m_{\lambda} =\begin{cases}
				n+2, & \mbox{ if } \lambda=(n) \\
				2(q+1), & \mbox{ if } \lambda=(p+q,p) \\
				q+1, & \mbox{ if } \lambda=(p+q,p,1) \\
				0, & \mbox{ in all other cases}
			\end{cases}.
		\end{equation*}
	\end{theorem}
	
	Remark that by Theorem \ref{base del Tideale D algebra} $UT_2^D\in\gv(UT_2)$ and $\gv(UT_2^D)$ grows exponentially, then it immediately follows that
	
	\begin{corollary}
		$\gv(UT_2)$ does not have almost polynomial growth of the codimensions.
	\end{corollary}
	
	Next, we shall prove that $\gv(UT_2^D)$ is a variety of $UT_2$-algebras of almost polynomial growth. If $\mathcal{V}$ is a variety of $\UT$-algebras, then for every $n\geq 1$, we write
	$$g\chi_{n}(\mathcal{V})=\sum_{\lambda\vdash n}m_{\lambda}^{\mathcal{V}}\chi_{\lambda},$$
	where $m_{\lambda}^{\mathcal{V}}$ denotes the multiplicity of irreducible character $\chi_{\lambda}$ in $g\chi_{n}(\mathcal{V})$.
	
	\begin{remark}\label{rmk: linear indipendet hwv}
		Recall that the $n+2$ linear independent generalized polynomials corresponding to the partition $ \lambda=(n)$ are:
		\begin{align*}
			& a(x)=x^n  ,\\
			& a_{22}^{(0)}(x)= e_{22} x^n ,\\
			& a_{22}^{(i)}(x)=x^{i-1} [x, e_{22}] x^{n-i} , \quad 1\leq i \leq n.
		\end{align*}
		The $2(q+1)$ linear independent generalized polynomials corresponding to the partition $ \lambda=(p+q,p)$ are:
		\begin{align*}
			&b^{(i)}_{p,q}(x,y)=x^{i}\underbrace{\bar{x}\cdots\tilde{x}}_{p-1}[x,y]\underbrace{\bar{y}\cdots\tilde{y}}_{p-1}x^{q-i}, \quad 0\leq i\leq q\\
			&d^{(i)}_{p,q}(x,y)=x^{i}\underbrace{\bar{x}\cdots\tilde{x}}_{p-1}(xe_{22}y-ye_{22}x)\underbrace{\bar{y}\cdots\tilde{y}}_{p-1}x^{q-i},\quad 0\leq i\leq q.
		\end{align*}
		Finally, the $q+1$ linear independent generalized polynomials corresponding to the partition $ \lambda=(p+q,p,1)$ are:
		\begin{equation*}
			h^{(i)}_{p,q}(x,y,z)=x^{i}\underbrace{\hat{x}\cdots\tilde{x}}_{p-1}\bar{x} \bar{y}\bar{z}\underbrace{\hat{y}\cdots\tilde{y}}_{p-1}x^{q-i},\quad 0\leq i\leq q. 
		\end{equation*}
		
	\end{remark}
	
	\begin{lemma}
		\label{LemPolynomialGrowth}
		Let $\mathcal{U}$ be a proper subvariety of $\gv(\UT^D)$. Then there exist constants $M<N$ such that
		$$x^{M}yx^{N-M}+\sum_{i<M}\mu_{i} x^{i}yx^{N-i} \in \GId(\mathcal{U}),$$
		for some  $\mu_{i}\in F$.
	\end{lemma}
	\begin{proof}
		Let $a(x),$ $a_{22}^{(i)}(x)$,  $b_{p,q}^{(j)}(x,y)$, $d_{p,q}^{(j)}(x,y)$, and $h_{p,q}^{(j)}(x,y,z)$, $0\leq i \leq n$, $0\leq j \leq q$, be the polynomial of Remark \ref{rmk: linear indipendet hwv}.
		Since $\mathcal{U}\varsubsetneqq \mathcal{V}=\gv(\UT^D)$, then there exists $\lambda\vdash n$ such that $m_{\lambda}^{\mathcal{U}}< m_{\lambda}^{\mathcal{V}}$. Thus by Theorem \ref{cocarattere di UT2D}, it follows that either
		\begin{equation}
			\label{Eq1riga}
			\alpha_1 a(x) + \sum_{i=0}^{n} \alpha_{2}^{(i)} a_{22}^{(i)}(x) \in  \GId(\mathcal{U}),
		\end{equation}
		with $ \alpha_1, \alpha_2^{(i)}$ not all zero, or
		\begin{equation}
			\label{Eq2righe}
			\sum_{i=0}^{q} \beta_i b_{p,q}^{(i)}(x,y)+ \sum_{i=0}^{q} \delta_i d_{p,q}^{(i)}(x,y) \in  \GId(\mathcal{U}),
		\end{equation}
		with $\beta_i, \delta_i$ not all zero, or
		\begin{equation}
			\label{Eq3riga}
			\sum_{i=0}^{q}\eta_i h_{p,q}^{(i)}(x,y,z)\in  \GId(\mathcal{U}),
		\end{equation}
		with $\eta_i$ not all zero.
		Suppose that \eqref{Eq2righe} holds. Then
		\begin{align*}
			f(x,y)=&\sum_{i=0}^{q}\beta_i x^{i}\underbrace{\bar{x}\cdots\tilde{x}}_{p-1}[x,y]\underbrace{\bar{y}\cdots\tilde{y}}_{p-1}x^{q-i}+ \sum_{i=0}^{q} \delta_i x^{i}\underbrace{\bar{x}\cdots\tilde{x}}_{p-1}(xe_{22}y-ye_{22}x)\underbrace{\bar{y}\cdots\tilde{y}}_{p-1}x^{q-i} \in   \GId(\mathcal{U})
			.
		\end{align*}
		If we substitute in $f(x,y)$ the variable $y$ with $y_1+y_2$, we obtain that
		\begin{align*}
			f(x,y_1,y_2)=&\sum_{i=0}^{q}\beta_i x^{i}\underbrace{\overline{x}\cdots\widetilde{x}}_{p-1}[x,y_1+y_2]\underbrace{\overline{(y_1+y_2)}\cdots\widetilde{(y_1+y_2)}}_{p-1}x^{q-i}\\
			&+ \sum_{i=0}^{q} \delta_i x^{i}\underbrace{\overline{x}\cdots\widetilde{x}}_{p-1}(xe_{22}(y_1+y_2)-(y_1+y_2)e_{22}x)\underbrace{\overline{(y_1+y_2)}\cdots\widetilde{(y_1+y_2)}}_{p-1}x^{q-i}\in \GId(\mathcal{U})
			.
		\end{align*}
		Now, let us consider in the polynomial $f(x,y_1,y_2)$ the component $f'(x,y_1,y_2)$ of degree $1$ in $y_2$. By substituting in $f'(x,y_1,y_2)$ the variable $y_1$ with $x^2$ and $y_2$ with $[x,y]$, we get that
		\begin{align*}
			&\sum_{i=0}^{q}\beta_i x^{i}\underbrace{\overline{x}\cdots\widetilde{x}}_{p-1}[x,[x,y]]\underbrace{\overline{x^2}\cdots\widetilde{x^2}}_{p-1}x^{q-i}+ \sum_{i=0}^{q} \delta_i x^{i}\underbrace{\overline{x}\cdots\widetilde{x}}_{p-1}(xe_{22}[x,y]-[x,y]e_{22}x)\underbrace{\overline{x^2}\cdots\widetilde{x^2}}_{p-1}x^{q-i}\in  \GId(\mathcal{U}).
		\end{align*}
		Since $e_{22}[x,y], \, [x,y]-[x,y]e_{22}\in \GId(UT_2^D)\subseteq\GId(\mathcal{U})$, it follows that
		\begin{equation}\label{eq: pol g}
			g(x,y)=\sum_{i=0}^{q}\beta_i \underbrace{\overline{x}\cdots\widetilde{x}}_{p-1}x[x,y]\underbrace{\overline{x^2}\cdots\widetilde{x^2}}_{p-1}x^{q-i}- \sum_{i=0}^{q} \gamma_i x^{i}\underbrace{\overline{x}\cdots\widetilde{x}}_{p-1}[x,y]x\underbrace{\overline{x^2}\cdots\widetilde{x^2}}_{p-1}x^{q-i}\in  \GId(\mathcal{U}),
		\end{equation}
		where $\gamma_i=\beta_i +\delta_i\in F$, $0\leq i \leq q$.
		
		Suppose first that $\beta_i\neq 0$ for some $0\leq i\leq q$, and let $t=\max \{i \, | \, \beta_i\neq 0$\} and $N'=\deg g(x,y)$. Since $g(x,y)\in \GId(\mathcal{U})$, we have that
		$$
		\beta_t x^{t+2p-1}[x,y]x^{N'-2p-t-1}+ \sum_{i<t+2p-1}\gamma_{i}^{'} x^{i}[x,y]x^{N'-i-2}\in \GId(\mathcal{U}),
		$$
		for some $\gamma_i^{'}\in F$. Since $\beta_t\neq 0$, we get that
		$$
		x^{t+2p}yx^{N'-2p-t-1}+\sum_{i<t+2p}\mu_{i} x^{i}yx^{N'-i-1}\in \GId(\mathcal{U}),
		$$
		for some $\mu_i\in F$. Now, if we set $N=N'-1$ and $M=t+2p$, then it follows that
		\begin{equation*}
			x^{M}yx^{N-M}+\sum_{i<M}\mu_{i} x^{i}yx^{N-i}\in \GId(\mathcal{U}),
		\end{equation*}
		for some $\mu_{i}\in F$, as required.
		
		Assume now that in \eqref{eq: pol g} $\beta_i=0$ for all $1\leq i \leq q$. Then $\gamma_i\neq 0$ for some $1\leq i \leq q$. So, let $t=\max \{i \, | \, \gamma_i \neq 0\}$ and $N'=\deg g(x,y)$ in \eqref{eq: pol g}. Then
		$$
		\gamma_t x^{t+2p-2}[x,y]x^{N'-2p-t}+ \sum_{i<t+2p-2}\gamma_{i}^{'} x^{i}[x,y]x^{N'-i-2}\in \GId(\mathcal{U}),
		$$
		for some $\gamma_i^{'}\in F$. Since $\gamma_t\neq 0$, we get that
		$$
		x^{t+2p-1}yx^{N'-2p-t-1}+\sum_{i<t+2p-1}\mu_{i} x^{i}yx^{N'-i-1}\in \GId(\mathcal{U}),
		$$
		for some $\mu_i\in F$. Now, if we set $N=N'-1$ and $M=t+2p-1$, then it follows that
		\begin{equation*}
			x^{M}yx^{N-M}+\sum_{i<M}\mu_{i} x^{i}yx^{N-i}\in \GId(\mathcal{U}),
		\end{equation*}
		for same $\mu_{i}\in F$, as required.
		
		Now, suppose that \eqref{Eq1riga} holds. Then, we have that
		\begin{equation}\label{eq:2righe cocharcter id}
			\alpha_1 x^n +  \alpha^{(0)}_{1} e_{22}x^n +\sum_{i=1}^{n} \alpha^{(i)}_{2} x^{i-1}[x,e_{22}] x^{n-i}\in  \GId(\mathcal{U}).
		\end{equation}
		Let us substitute $x$ with $x_1+x_2$ in \eqref{eq:2righe cocharcter id}, and consider the homogeneous component of degree $1$ in $x_2$. Then in this homogeneous component, we substitute $x_1$ with $x$ and $x_2$ with $[x,y]$. Thus, with similar computations as in the previous case, we reach the desired conclusion.
		
		Finally, suppose that \eqref{Eq3riga} holds in $\mathcal{U}$. By substituting in $ h_{p,q}^{(i)}(x,y,z)$ the variable $z$ with $x^2$, we obtain \eqref{Eq2righe}, and, by the first case, the proof is complete.
	\end{proof}

	\begin{proposition}\label{molteplicita limitate}
		Let $\mathcal{U}$ be a proper subvariety of $\gv(\UT^D)$. Then there exists a constant $\bar{N}$ such that $m_{\lambda}^{\mathcal{U}}\leq \bar{N}$ for any $\lambda\vdash n$, $n\geq 1$.
	\end{proposition}
	\begin{proof}
		By Lemma \ref{LemPolynomialGrowth},  there exists $N$ such that
		\begin{equation}
			\label{EqNandM}
			x^{M}yx^{N-M}+\sum_{i<M}\mu_{i} x^{i}yx^{N-i} \in \GId(\mathcal{U}),
		\end{equation}
		for some $\mu_{i}\in F$ and a suitable $M<N$. We shall prove that $m_{\lambda}^{\mathcal{U}}\leq 2N$ for all $\lambda\vdash n$. By Theorem \ref{cocarattere di UT2D} it is enough to consider the cases when either $\lambda=(n)$, or $\lambda=(p+q,p)$, or $\lambda=(p+q,p,1)$.
		
		We prove the statement for $\lambda=(p+q,p)$. The other cases will follow with similar arguments.
		
		\noindent If $q < N$ there is nothing to prove. So, let us assume that $q\geq N$ and consider the polynomials $b^{(i)}_{p,q}(x,y)$ and $d^{(i)}_{p,q}(x,y)$, $0\leq i \leq q$, defined in Remark \ref{rmk: linear indipendet hwv}. 
		Notice that from relation \eqref{EqNandM} it follows that
		\begin{equation}\label{eq: MandN commutatore}
			x^{M}\underbrace{\bar{x}\cdots\tilde{x}}_{p-1}[x,y]\underbrace{\bar{y}\cdots\tilde{y}}_{p-1}x^{N-M}\equiv\sum_{i<M}\mu_{i} x^{i}\underbrace{\bar{x}\cdots\tilde{x}}_{p-1}[x,y]\underbrace{\bar{y}\cdots\tilde{y}}_{p-1}x^{N-i} \ (\md\GId(\mathcal{U})),
		\end{equation}
		and
		\begin{equation}\label{eq: MandN e_22}
			x^{M}\underbrace{\bar{x}\cdots\tilde{x}}_{p-1}(xe_{22}y-ye_{22}x)\underbrace{\bar{y}\cdots\tilde{y}}_{p-1}x^{N-M}\equiv\sum_{i<M}\mu_{i} x^{i}\underbrace{\bar{x}\cdots\tilde{x}}_{p-1}(xe_{22}y-ye_{22}x)\underbrace{\bar{y}\cdots\tilde{y}}_{p-1}x^{N-i} \ (\md\GId(\mathcal{U})).
		\end{equation}
		Hence, since $q \geq N$, we can apply the relation \eqref{eq: MandN commutatore} to any polynomial $b^{(i)}_{p,q}(x,y)$
		such that $i\geq M$, and, as a consequence, we get that
		$$b^{(i)}_{p,q}(x,y)\equiv\sum_{j<M}b^{(j)}_{p,q}(x,y) \ (\md\GId(\mathcal{U})).$$
		Similarly, since $q \geq N$, we can apply the relation \eqref{eq: MandN e_22} to any polynomial $d^{(i)}_{p,q}(x,y)$ such that $i\geq M$, and we obtain that
		$$d^{(i)}_{p,q}(x,y)\equiv\sum_{j<M}d^{(j)}_{p,q}(x,y) \ (\md\GId(\mathcal{U})).$$
		Therefore, it follows that $m_{\lambda}^{\mathcal{U}}\leq 2M \leq 2N=\bar{N}$, as required.
	\end{proof}

	\begin{theorem}\label{Thm: UT2D APG}
		The variety of $UT_2$-algebras generated by $\UT^D$ has almost polynomial growth.
	\end{theorem}
	\begin{proof}
		Let $\mathcal U$ be a proper subvariety of $\mathcal V = \gv(UT_2^D).$ We shall prove that $\mathcal U$ has polynomial growth of the codimensions. By Lemma \ref{LemPolynomialGrowth}, there exist constant $M<N$ such that
		$$
		x^Myx^{N-M} + \sum_{i<M}\mu_i x^iyx^{N-i}\in\GId(\mathcal U)
		$$
		for some $\mu_i\in F.$ By a standard multilinearization process (see for instance \cite[Theorem 1.3.8]{GiambrunoZaicevbook}), we get
		$$
		\sum_{\sigma\in S_N} x_{\sigma(1)}\cdots x_{\sigma(M)}yx_{\sigma(M+1)}\cdots x_{\sigma(N)} + \sum_{i<M}\sum_{\sigma\in S_N}\mu_i x_{\sigma(1)}\cdots x_{\sigma(i)}yx_{\sigma(i+1)}\cdots x_{\sigma (N)}\in\GId(\mathcal U)
		$$
		where $x_1, \ldots, x_N$ are new variables.
		
		In the previous identity, we substitute $y$ by $[y_1,y_2],$ we multiply on the right by $z_1\cdots z_M$ and we alternate $x_i$ with $z_i,$ for all $1\leq i\leq M.$ Since $[x_1,x_2][x_3,x_4]\in \GId(UT_2^D) \subseteq \GId(\mathcal{U}),$ it follows that
		$$ 
		\bar{x}_1\cdots \tilde{x}_M[y_1,y_2]\bar{z}_1\cdots \tilde{z}_Mx_{M+1}\cdots x_{N}\in\GId(\mathcal U).
		$$
		Now, we multiply on the left by $z_{M+1}\cdots z_N$ and we alternate $x_j$ with $z_j$ for all $M+1\leq j\leq N.$ It readily follows that
		\begin{equation}\label{uccidi il primo hwv}
			\bar{x}_1\cdots \tilde{x}_N[y_1,y_2]\bar{z}_1\cdots \tilde{z}_N\in\GId(\mathcal U).
		\end{equation}
		Take the previous identity and substitute firstly $y_1$ by $y_1e_{22}$ and, secondly, $y_2$ by $y_2e_{22}.$ We get
		\begin{equation*}
			\begin{split}
				&\bar{x}_1\cdots \tilde{x}_N(y_1e_{22}y_2-y_2y_1e_{22})\bar{z}_1\cdots \tilde{z}_N\in \GId(\mathcal{U}), \\
				&\bar{x}_1\cdots \tilde{x}_N(y_1y_2e_{22}-y_2e_{22}y_1)\bar{z}_1\cdots \tilde{z}_N\in\GId(\mathcal{U}) .
			\end{split}
		\end{equation*}
		Let us sum the previous identities and, since $[x_1,x_2]-[x_1,x_2]e_{22}\in\GId(UT_2^D)\subseteq\GId(\mathcal U)$ and \eqref{uccidi il primo hwv} holds, we obtain
		$$
		\bar{x}_1\cdots \tilde{x}_N(y_1e_{22}y_2-y_2e_{22}y_1)\bar{z}_1\cdots \tilde{z}_N\in \GId(\mathcal{U}).
		$$
		By renaming the variables, we get
		\begin{equation}\label{uccidi il secondo hwv}
			\bar{x}_1\cdots \tilde{x}_N \hat{x}_{N+1}e_{22}\hat{z}_{N+1}\bar{z}_1\cdots \tilde{z}_N\in\GId(\mathcal U).
		\end{equation}
		The identities $\eqref{uccidi il primo hwv}$ and $\eqref{uccidi il secondo hwv}$ tell us that the irreducible $S_{2(N+1)}$-character corresponding to the partition $\lambda= (N+1,N+1)$ participates into the $2(N+1)$th generalized cocharacter of $\mathcal U$ with a zero multiplicity, i.e., $m_{(N+1,N+1)}^{\mathcal U}=0.$
		
		Finally, take identity \eqref{uccidi il primo hwv}, multiply it on the right by $y_{N+1}$ and alternate $y_1,$ $y_2$ and $y_{N+1}.$ By renaming as before the variable $y_1$ by $x_{N+1}$ and $y_2$ by $z_{N+1},$ we get
		$$
		\bar{x}_1\cdots \tilde{x}_N \hat{x}_{N+1}\hat{y}_{N+1}\hat{z}_{N+1}\bar{z}_1\cdots \tilde{z}_N\in\GId(\mathcal U).
		$$
		Thus, as in the previous case, $m_{(N+1,N+1,1)}^{\mathcal{U}}=0.$
		
		Hence, if $\lambda\vdash n$ is such that $\lambda_2\geq N+2$ then $m_{\lambda}^{\mathcal U}=0$ or, equivalently, if $\chi_\lambda$ appears with a non-zero multiplicity in the generalized $S_n$-cocharacter of $\mathcal U,$ then $\lambda$ must contain at most $N+1$ boxes below the first row. Therefore
		$$
		g\chi_n(\mathcal U) = \sum_{\substack{\lambda\vdash n \\ |\lambda|- \lambda_1\leq N+1}}m_{\lambda}^{\mathcal U}\chi_\lambda.
		$$
		Recall that $\lambda_i$ stands for the number of boxes of the $i$th row of $\lambda.$
		
		Since $|\lambda|-\lambda_1\leq N+1,$ then $\lambda_1\geq n-(N+1)$ and by the hook formula
		$$
		d_\lambda = \chi_\lambda(1) \leq \frac{n!}{(n-(N+1))!} \leq n^{N+1}.
		$$
		We are now in a position to reach the goal, in fact by the previous remark and by Proposition \ref{molteplicita limitate}
		$$
		gc_n(\mathcal{U}) = g\chi_n(\mathcal U)(1) = \sum_{\substack{\lambda\vdash n \\ |\lambda|- \lambda_1\leq N+1}}m_{\lambda}^{\mathcal U}d_\lambda\leq  \sum_{\substack{\lambda\vdash n \\ |\lambda|- \lambda_1\leq N+1}} \bar{N} n^{N+1}\leq (N+1)^2N'n^{N+1},
		$$
		since the number of partitions such that $|\lambda| - \lambda_1\leq N+1$ is bounded by $(N+1)^2.$ Therefore $gc_n(\mathcal{U})$ is polynomially bounded and we are done.
	\end{proof}

	Let us denote by $UT_2^F$ the $F$-algebra $UT_2$ regarded as $UT_2$-algebra, i.e., $UT_2^F$ has a structure of $UT_2$-bimodule where $1:=1_{UT_2}$ acts by left and right by multiplication, $e_{22}\cdot a = a\cdot e_{22}=0$ and
	$
	e_{12}\cdot a = a\cdot e_{12}=0
	$ 
	for all $a\in UT_2.$ Clearly, from the definition of this new action it readily follows that $e_{22}x\equiv 0$, $xe_{22}\equiv 0$, $e_{12}x\equiv 0$ and $xe_{12}\equiv 0$ are generalize identities of $UT_2^F$. Thus, we are dealing with ordinary polynomial identities, and by the results in \cite{Malcev1971,BenantiGiambrunoSviridova2004,Kemer1979} we have the following.
	
	\begin{theorem}\label{Thm: Identità ordinarie}
		Let $UT_2^F$ be the $UT_2$-algebra with the above action. Then $\GId(UT_2^F)$ is generated, as $T_{UT_2}$-ideal,  by the following polynomials:
		$$
		e_{22}x; \quad xe_{22}; \quad [x_1,x_2][x_3,x_4].
		$$
		Moreover, $gc_n(UT_2^F) = 2^{n-1}(n-2)+2.$
	\end{theorem}
	
	\begin{theorem}
		Let $g\chi_n(\UT^F)=\sum_{\lambda\vdash n} m_{\lambda} \chi_{\lambda}$ be the $n$th generalized cocharacter of $\UT^F.$ Then
		\begin{equation*}
			m_{\lambda} =\begin{cases}
				1, & \mbox{ if } \lambda=(n) \\
				q+1, & \mbox{ if } \lambda=(p+q,p) \mbox{ or }  \lambda=(p+q,p,1) \\
				0, & \mbox{ in all other cases}
			\end{cases}.
		\end{equation*}
	\end{theorem}
	
	\begin{theorem}\label{Thm: ordinary identities APG}
		The variety of $UT_2$-algebras generated by $\UT^F$ has almost polynomial growth.
	\end{theorem}
	
	Notice that from Theorems \ref{basedelTideale} and \ref{Thm: Identità ordinarie} it follows that $UT_2^F\notin \gv(UT_2)$. Also, as a consequence of Theorems \ref{base del Tideale D algebra} and \ref{Thm: Identità ordinarie} we have that $\gI(UT_2^F)\nsubseteq \gI(UT_2^D)$ and $\gI(UT_2^D)\nsubseteq \gI(UT_2^F)$. Thus by Theorems \ref{Thm: UT2D APG} and \ref{Thm: ordinary identities APG} we have the following.

	\begin{corollary}
		\label{cor:almost polynomial growth}
		The algebras $UT_2^F$ and $UT_2^D$ generate two distinct varieties of $UT_2$-algebras of almost polynomial growth.
	\end{corollary}




	


\begin{thebibliography}{99}
		%
		
		\bibitem{Amitsur1965} S.A. Amitsur, {\em Generalized polynomial identities and pivotal monomials}, Trans. Amer. Math. Soc. {\bf 114} (1965), 210--226.
		
		
		
		\bibitem{BeidarMartindaleMikhalevm1996} K.I. Beidar, W.S. Martindale, A.V. Mikhalev, Rings with generalized identities, Monographs and Textbooks in Pure and Applied Mathematics, Marcel Dekker, Inc. (1996).
		
		\bibitem{BenantiGiambrunoSviridova2004} F. Benanti, A. Giambruno, I. Sviridova, {\em Asymptotics for the multiplicities in the cocharacters of some PI-algebras}, Proc. Amer. Math. Soc. {\bf 132} (2004), no. 3, 669--679.
		
		%
		
		\bibitem{BresarSpenko2015}M. Bre\v{s}ar, \v{S}. \v{S}penko, {\em Functional Identities on Matrix Algebras}. Algebr. Represent. Theor. {\bf 18} (2015), 1337--1356.
		%
		%
		%
		\bibitem{Drenskybook} V. Drensky, {\em Free algebras and PI-algebras, graduate course in algebra}, Springer, Singapore (2000).
		%
		%
		%
		%
		%
		%
		%
		%
		%
		%
		%
		%
		%
		%
		%
		%
		\bibitem{GiambrunoRizzo2019} A. Giambruno, C. Rizzo, {\em Differential identities, $2\times 2$ upper triangular matrices and varieties of almost polynomial growth}, J. Pure Appl. Algebra {\bf 223} (2019), no. 4, 1710--1727.
		%
		%
		%
		%
		%
		%
		\bibitem{GiambrunoZaicevbook} A. Giambruno, M. Zaicev, {\em Polynomial identities and asymptotic methods}, AMS, Math. Surv. Monogr. {\bf 122} (2005).
		%
		%
		%
		%
		
		\bibitem{Gordienko2010} A.S. Gordienko, {\em Codimensions of generalized polynomial identities}, Sb. Math. {\bf 201} (2010), 235--251.
		
		%
		%
		%
		%
		
		\bibitem{JamesKerber1981} G. James, A. Kerber, The representation theory of the symmetric group, Addison-Wesley, London, 1981.
		
		%
		\bibitem{Kharchenko1975} V.K. Kharchenko, {\em Generalized identities with automorphisms}, Algebra and Logic {\bf 14} (1975), 132--148.
		
		\bibitem{Kharchenko1978} V.K. Kharchenko, {\em Differential identities of prime rings}, Algebra and Logic {\bf 17} (1978), 155--168. 
		
		\bibitem{Kharchenko1979}  V.K. Kharchenko, {\em Differential identities of semiprime rings}, Algebra Log. {\bf 18} (1979) 86-–119.
		%
		%
		%
		\bibitem{Kemer1979} A.R. Kemer, {\em Varieties of finite rank}, Proc. 15-th All the Union Algebraic Conf., Krasnoyarsk. {\bf 2} (1979), p. 73 (in Russian).
		%
		%
		%
		%
		%
		%
		%
		%
		%
		%
		%
		%
		%
		%
		\bibitem{Malcev1971} J.N. Malcev, {\em A basis for the identities of the algebra of upper triangular matrices}, Algebra i Logika {\bf 10} (1971), 393--400.
		%
		%
		\bibitem{Martindale1969} W.S. Martindale 3rd, {\em Prime rings satisfying a generalized polynomial identity}, J. Algebra {\bf 12} (1969), 576--584.
		
		\bibitem{Martindale1972} W.S. Martindale 3rd, {\em Prime rings with involution and generalized polynomial identities}, J. Algebra {\bf 22} (1972),
		502--516.
		
		\bibitem{Martino2019} F. Martino, {\em Varieties of special Jordan algebras of almost polynomial growth}, J. Algebra {\bf 531} (2019), 184--196.
		%
		%
		\bibitem{MishchenkoValenti2000} S. Mishchenko and A. Valenti, {\em A star-variety with almost polynomial growth}, J. Algebra {\bf 223} (2000), no. 1, 66--84.
		%
		%
		%
		%
		\bibitem{Regev1972} A. Regev, {\em Existence of identities in $A\otimes B$}, Israel J. Math. {\bf 11} (1972), 131--152.
		%
		%
		%
		%
		%
		
		\bibitem{Rowen1975} L.H. Rowen, {\em Generalized polynomial identities}, J. Algebra {\bf 34} (1975), 458--480.
		
		\bibitem{Rowen1976} L.H. Rowen, {\em Generalized polynomial identities. II}, J. Algebra {\bf 38} (1976), 380--392.
		
		\bibitem{Rowen1977} L.H. Rowen, {\em Generalized polynomial identities. III}, J. Algebra {\bf 46} (1977), 305--314.
		%
		%
		\bibitem{Valenti2002} A. Valenti, {\em The graded identities of upper triangular matrices of size two}, J. Pure Appl. Algebra {\bf 172} (2002), 325--335.
		%
		%
	\end{thebibliography}
\end{document}